\documentclass[reqno]{article}
\usepackage{amssymb, amsmath, amsthm, amsfonts, amscd}
\usepackage[english]{babel}

\usepackage{stmaryrd}

\usepackage{multicol}
\usepackage{graphicx}
\usepackage{subcaption}

\usepackage{epstopdf}

\usepackage{dsfont}
\usepackage{breqn}
\usepackage{mathtools}
\usepackage{color}
\usepackage{hyperref}
\hypersetup{
    colorlinks=true,
    citecolor=red,
    filecolor=black,
    linkcolor=blue,
    urlcolor=blue
}

\usepackage{algorithm}
\usepackage[noend]{algpseudocode}

\chardef\bslash=`\\ 

\newtheorem{theorem}{Theorem}[section]
\newtheorem{corollary}[theorem]{Corollary}

\newtheorem{proposition}[theorem]{Proposition}

\newcommand{\RR}{\mathcal{R}}

\newcommand{\Z}{\mathbb{Z}}
\newcommand{\Q}{\mathbb{Q}}
\newcommand{\R}{\mathbb{R}}

\newcommand{\C}{\mathbb{C}}

\newcommand{\T}{\mathbb{T}}

\newcommand{\Id}{\mathrm{Id}}

\newcommand{\Conj}{\mathrm{Conj}}
\newcommand{\NF}{\mathrm{NF}}

\newcommand{\phd}{\cite{NKPhD}}

\def\a{\alpha }
\def\g{\gamma }
\def\sm{C^{\infty} }

\def\u{\theta}

\def\t{\tau}

\def\w{\omega}

\def\.{\cdot }
\def\ra{\rightarrow}

\title{
\textsc{\textbf{A note on cocycles in $\mathbb{T}\times SO(3)$}}\\
\author{Nikolaos Karaliolios}
}

\begin{document}

\maketitle

\begin{abstract}
This short note studies $C^{\infty}$-smooth cocycles in $\mathbb{T}\times SO(3)$ that have $0$
degree and are non-homotopic to constants. The study picks up from where the author's PhD thesis
left the subject, and shows that, under a relevant and full measure arithmetic condition, such cocycles can
be conjugated to a simple model. Moreover, under the same arithmetic condition, the cocycle
can be conjugated arbitrarily close to constant cocycles by a $2$-periodic conjugation.
\end{abstract}  

\section{Introduction}

In order to keep this note short and avoid repetition, we will cite heavily our PhD thesis,
\phd, but refer to the arXiv version, available \href{https://arxiv.org/abs/1505.04562}{here},
in order to facilitate the reader. We will refer to works earlier than our PhD thesis on the subject
mainly through reference to \phd, since this note is an appendix to it, and is not
intended to be submitted to a journal. Any lack of citations is due to the author lacking
full access to journals.

The homotopy group of $SO(3)$ is naturally isomorphic to $\Z / 2\Z$. A model for a path in the
non-trivial homotopy class is the rotation by an angle in $[0,2\pi]$ around the z-axis,
\begin{equation} \label{eqgeodesic}
R_{2\pi x} =
    \begin{pmatrix}
        \cos (2\pi x) & -\sin (2\pi x) & 0 \\
        \sin (2\pi x) & \cos (2\pi x) & 0 \\
        0 & 0 & 1
    \end{pmatrix}, x \in [0,1]
\end{equation}
This model corresponds to $E_{1/2}(\.): [0,1] \ra SU(2)$ in the notation of \phd, see p. 92-93, section
"Cocycles in $SU(2)$ and $SO(3)$". This path in $SU(2)$ emanates from the $\Id$ and reaches $-\Id$ in time
$1$, thus not a loop in $SU(2)$. Since, however, $SO(3) \equiv SU(2)/ \{\pm \Id \}$, $E_{1/2}(\.): [0,1] \ra SU(2)$
maps to a closed loop in $SO(3)$ that is non-homotopic to the $\Id$.

The degree is related with the linear growth of the derivatives of the cocycle, see
def. 6.2, as well as theorem 6.18 of \phd\, and references therein for a connection with the theory of cocycles in
$SL(2, \C)$.\footnote{It is easy to see that our calculations generalize in higher dimensions and that the quantization
of the acceleration as proved in \cite{avila2014complex} is the same as the one of the degree as proved in \phd.} A
cocycle over any irrational rotation with
the fiber-wise dynamics defined by eq. \eqref{eqgeodesic} has linear growth and thus non-zero degree.
The existence of $0$-degree cocycles non-homotopic to the $\Id$ is delicate to spot, but an explicit
example can be given by
\begin{equation}
    (\a, A R_{2\pi \.})
\end{equation}
where $\a$ is minimal and $A$ anti-commutes with $R_{2\pi \.}$.

The following arithmetic condition is a variant of the $RDC$ condition, see def. 2.13 of \phd.
Let $\widetilde{DC}(\g, \t)$ be the set of $\a \in \T$ such that
\begin{equation}
    |k\a -1/2|_{\Z} \geq \frac{\g^{-1}}{|k|^{\t}}.
\end{equation}
A simple calculation shows that if $\a \in \widetilde{DC}(\g, \t)$, then $2\a \in DC(\g/2, \t)$, where
$DC(\g, \t)$ is defined as the set of $\a \in \T$ such that
\begin{equation}
    |k\a |_{\Z} \geq \frac{\g^{-1}}{|k|^{\t}}.
\end{equation}

Let also $\mathrm{G}$ be the Gauss map $\a \mapsto \{ \a ^{-1} \}$ where $\{ \. \}$ stands for the fractional part.
Then, we call $\widetilde{RDC}(\g, \t)$ the set of $ \a \in \T \setminus \Q$ such that
$\mathrm{G}^{n}(\a ) \in \widetilde{DC}(\g, \t)$ for infinitely many $n$. This condition is of full Haar measure
for the same reasons as $RDC$. We also note that $\a \in \widetilde{RDC}(\g, \t)$ implies $2\a _n \in DC(\g/2, \t)$
infinitely often, where the continued fractions are those associated to $\a \pmod 1$ and not to $2\a \pmod 1$.

In \S \ref{secproof} we will sketch a proof of the following theorem. We will use notation from \S
2.1.2 and from \S 6, esp. \S 6.3, of \phd.
\begin{theorem} \label{mainthm}
    Let $(\a, A(\.))$ be a $\sm $-smooth cocycle in $\T \times SO(3)$, of $0$ degree and non-homotopic to the
    $\Id$. Suppose, moreover, that $\a \in \widetilde{RDC}$. Then, $(\a, A(\.))$ is $\sm$ almost reducible by
    $2$-periodic conjugations.
\end{theorem}
We note that $2$-periodic conjugations arise naturally in the proof. Applying the renormalization machinery
directly to the second iterate of the cocycle avoids introducing conjugations of longer periods, because
all the obvious obstructions to almost reducibility (both homotopy and degree) vanish.
The known procedure shows that, for such a $\sm$ smooth cocycle, and if $2\a \in RDC $, then the cocycle is
almost reducible.

In \cite{HOU2024109943}, the authors proved a semi-local Almost Reducibility for cocycles over all irrational rotations
in compact groups in the real analytic category (denoted by $C^{\w}$), see theorem 3.1 of the reference. Assuming
convergence of renormalization of $0$-degree cocycles to constant actions irrespective of homotopy, and $C^{\w}$
normalization of constant actions as in \S \ref{secproof} of this note, the $\widetilde{RDC}$ arithmetic condition
of theorem \ref{mainthm} can be lifted for $C^{\w}$ smooth cocycles by coupling
the semi-local Almost Reducibility theorem with standard arguments.

The techniques developed in \cite{NKRotVec} seem to be applicable in the $\sm$ category without any arithmetic condition,
and to arbitrary cocycles satisfying a universal closeness-to-constants condition such as the applicability of the Campbell–Hausdorff
formula. We will come back to a more systematic study of the techniques developed in \cite{NKRotVec} in a forthcoming article,
\cite{NKRotVecAll}. We expect to obtain an extension of the definition of the rotation vector as obtained in \cite{NKRotVec}
to all quasi-periodic cocycles in compact groups over minimal rotations in an open set containing constant cocycles, and to
address the classification of $0$-degree cocycles that are homotopic to the $\Id$ up to conjugation.

The proof of theorem \ref{mainthm} will use the following proposition. It is not necessary, but obtains a normal form for
the renormalization representatives without excessive restriction with respect to the arithmetics.

\begin{proposition} \label{mainprop}
    Let $(\a, A(\.))$ be a $\sm $-smooth cocycle in $\T \times SO(3)$, of $0$ degree and non-homotopic to the
    $\Id$. Suppose, moreover, that $\a \in \widetilde{RDC}$. Then, $(\a, A(\.))$ has renormalization representatives
    $\sm$ of the form
    \begin{equation} \label{eqnormformrenorm}
        (\a _n , A E_{1/2}(\.)\exp (\{0,z_n\}_{su(2)}))
    \end{equation}
    where $z_n \in \C$ can be arbitrarily small, and $A$:
        \begin{enumerate}
        \item anti-commutes with $E_{1/2}(\.)$: $Ad(A).E_{1/2}(\.) = E_{1/2}(-\.) $, and
        \item acts on $z_n$ by complex conjugation:
        $Ad(A).\{0,z_n\}_{su(2)} = \{0,\bar{z}_n\}_{su(2)}$
    \end{enumerate}
\end{proposition}
This shows that $0$ degree cocycles that are non-homotopic to the $\Id$ form, under a full measure condition on the
frequency, a one-parameter family in a way that is analogous to the rigidity of regular cocycles in compact groups,
see \S 8 of \phd. We will refer to the expression of eq. \eqref{eqnormformrenorm} as a normal form.

The second iterate of the renormalization representative of cocycles satisfying the assumptions of the theorem
is a cocycle with a closed and remarkably simple form. The proof is by direct calculation.
\begin{corollary}
    The second iterate of the normal form of eq. \eqref{eqnormformrenorm} is
    of the form
    \begin{equation}
        (2\a_n ,E_{1/2}(\a )\exp (\{0,\bar{z}_n\}_{su(2)}).\exp (\{0,e^{2i\pi \. }z_n\}_{su(2)}))
    \end{equation}
    with $z_n$ arbitrarily small.
\end{corollary}

We remind that or the purpose of proving
theorem \ref{mainthm}, closeness to constants is sufficient and does not require this closed form expression.
Closeness to constants follows directly from the convergence of renormalization to constant actions and the uniformity
of the normalization operator and the fact that the second iterate is studied, without any need for longer iterations.

This form is close to the KAM-normal form as it was used in \S 4 of \cite{NKRotVec}. The difference
is in the term $\exp (\{0,\bar{z}_n\}_{su(2)})$ which is non-diagonal, and, as a consequence, the reduction scheme
of the reference can be applied with explicit calculations only once in order to reduce the cocycle to a perturbation with
Fourier spectrum in $[-2,2]$ and size $O(|z_n|^2)$. In the absence of the term
$\exp (\{0,\bar{z}_n\}_{su(2)})$, the secondd iterate would actually be $\sm$
reducible with an explicit conjugation. In any case, the following corollary is immediate.
\begin{corollary} \label{cormain}
    Such a cocycle admits infinitely many renormalization representatives whose second iterate is almost reducible by
    sequences of $1$-periodic conjugations.
\end{corollary}

Assuming proposition \ref{mainprop}, we can now prove theorem \ref{mainthm}.
\begin{proof}[Proof of theorem \ref{mainthm}]
    By proposition \ref{mainprop}, the cocycles satisfying the hypotheses of the proposition admit renormalization
    representatives in Normal Form. Suppose that this can be achieved after one step of renormalization, the general
    case following directly. Then, the action associated to the cocycles itself, i.e.
    \begin{equation}
         \Phi = \begin{pmatrix}
             (1, \Id) \\
             (\a , A (\. ))
         \end{pmatrix}
    \end{equation}
    is transformed by the $SL(2,\Z)$ matrix
    \begin{equation}
         \mathrm{CF}(\a) = \begin{pmatrix}
             a & 1 \\
             1 & 0 
         \end{pmatrix}
         =
         \begin{pmatrix}
             \lfloor \a ^{-1} \rfloor & 1 \\
             1 & 0 
         \end{pmatrix}
    \end{equation}
    to the action
    \begin{equation}
         \RR_{1} \Phi = \mathrm{CF}(\a)_{*}\Phi = \begin{pmatrix}
             (\a , A (\. )) \\
             (1-a\a , A_{-a}(\. ))
         \end{pmatrix}
    \end{equation}
    We note that we have omitted the change of scale for simplicity. Let $B: \R \ra SO(3)$ be the conjugation that
    normalizes the action $\RR_{1} \Phi$ to the normal form, denoted by $\NF (\.) $:
    \begin{equation}
        \Conj _{B}\RR_{1} \Phi = \begin{pmatrix}
             (\a , \Id) \\
             (1-a\a , \NF (\. ))
         \end{pmatrix}
    \end{equation}
    where $B(\. + \a).A(\.).B^{*}(\.) = \Id$, $\NF (\. )$ is $\a$-periodic, and, actually, by a simple calculation,
    \begin{equation}
        NF(\.) = B(\. +1 -n\a)B^{*}(\. -n\a)= B(\. +1 )B^{*}(\. )
    \end{equation}
    This last expression can be verified, as a sanity check, to be $\a$-periodic as a consequence of the $1$-periodicity
    of $A(\. )$ and the normalization property of $B(\. )$.
    
    By corollary \ref{cormain}, the cocycle
    \begin{equation}
        (2-2a\a , \NF_2 (\. )) = (2-2a\a , \NF (\. + 1-a\a)\NF (\. )) = (2-2a\a , \NF (\. + 1)\NF (\. ))
    \end{equation}
    is almost reducible by $\a$-periodic conjugations.
    Suppose that $D_m(\.)$ is such a conjugation satisfying
    \begin{equation}
        D_{m} (\.+2-2a\a ) \NF_{2} (\. ) D_{m}^{*}(\.) = A_m \exp (U_m(\.))
    \end{equation}
    with $U_m(\.)$ arbitrarily small. By the $\a$-periodicity of $D_m(\.)$, it also holds that
    \begin{equation}
        D_m(\.+ 2) \NF_{2} (\. ) D_m^{*}(\.) = A_m \exp (U_m(\.))
    \end{equation}
    Consider, now, the action
    \begin{equation}
        \Conj _{D_m.B}\RR_{1} \Phi = \begin{pmatrix}
             (\a , \Id) \\
             (1-a\a , D_m(\.+ 1) \NF (\. ) D_m^{*}(\.))
         \end{pmatrix}
    \end{equation}
    Inversion of the first step of renormalization yields the action
    \begin{equation}
       \mathrm{CF}(\a)^{*} \Conj _{D_m.B}\RR_{1} \Phi = \begin{pmatrix}
             (1 ,  D_m(\.+ 1) \NF (\. ) D_m^{*}(\.)) \\
             (\a , \Id)
         \end{pmatrix}
    \end{equation}
    Passing on to two-periodic actions gives the action generated by
    \begin{equation}\begin{pmatrix}
             (2 ,  D_m(\.+ 2) \NF_2 (\. ) D_m^{*}(\.)) \\
             (\a , \Id)
         \end{pmatrix}
         =
         \begin{pmatrix}
             (2 ,  A_m \exp (U_m(\.))) \\
             (\a , \Id)
         \end{pmatrix}
    \end{equation}
    This action can be normalized to an action that is arbitrarily close to a constant, thus proving the theorem.
\end{proof}

It is noted that if one applies the matrix $\mathrm{CF}(\a)$ to the action generated by
    \begin{equation}
         \begin{pmatrix}
             (1, \Id) \\
             (2\a , A_2 (\. ))
         \end{pmatrix}
    \end{equation}
which would seem a natural idea for studying the second iterate of the cocycle, and then conjugate by
the normalizing conjugation $B(\.)$, they obtain the action
    \begin{equation}
        \begin{pmatrix}
             (2\a , \Id )\\
             (1-2n\a , \NF (\. ))
         \end{pmatrix}
    \end{equation}
which is quite remarkable. Following the calculations of the proof and inverting the renormalization only
yields $2$-periodic almost reducibility for the second iterate, which is the flavor of the results obtained
in \phd. The proof of theorem \ref{mainthm} can be applied to all results obtained in \phd \,giving the same
type of improvement: if $m$ iterations are needed to lift the obstruction due to homotopy, then the cocycle
is almost reducible to constant cocycles or the relevant normal form $\bmod m$.

\textbf{Acknowledgement} The author would like to kindly thank Xuanji Hou, Yi Pan and Qi Zhou for reminding him the subject. 

\section{Proof of proposition \ref{mainprop}} \label{secproof}

\subsection{The model of the dynamics}
The heavy lifting for obtaining the model of the dynamics was done in \S 6 of \phd.
For cocycles of $0$ degree, renormalization converges uniformly to constant actions, and the constants
of the actions commute by lemma 6.26 of \phd. In order to fix notation, let the constants at the step
$n$ be
\begin{equation}
    \begin{pmatrix}
        (1,C_n)\\
        (\a_n , A_n)
    \end{pmatrix}
\end{equation}
The constants cannot be on the same maximal torus. Assuming one constant to be diagonal
in $SO(3)$, which can be obtained by applying a constant conjugation, the other constant is diagonal too,
and, again up to a constant conjugation, the model is $C_n = \mathrm{diag}(-1,-1,1)$, and
$A_n = \mathrm{diag}(1,-1,-1)$. Conjugation by $R_{-2\pi \. /2}$ and use of the anti-commutation relation gives an
action which is a perturbation of
\begin{equation}
    \begin{pmatrix}
        (1,\Id)\\
        (\a_n , R_{-2\pi \a_n /2}A_n R_{2\pi \.})
    \end{pmatrix}
\end{equation}
Normalization gives the cocycle
\begin{equation}
    (\a_n , R_{-2\pi \a_n /2}A_n R_{2\pi \.} \exp(U_n(\.)))
\end{equation}

We will fix such a cocycle, assuming that $\a \in \widetilde{RDC}$, and drop the subscript $n$ for the remaining
of the proof. For ease of notation in the calculations, we will use the $SU(2)$ version of the cocycle, which
reads, with subscripts dropped,
\begin{equation}
    (\a , A E_{1/2}(\. + \a/2)\exp (U(\. )))
\end{equation}
where
\begin{align}
    A &= \begin{pmatrix}
        0 & 1 \\
        -1 & 0
        \end{pmatrix}
    \\
        E_{r} (\.) &= \begin{pmatrix}
        e^{2i\pi \.} & 0 \\
        0 & e^{2i\pi \.}
    \end{pmatrix}
    \\
    U(\. ) &= \{ U_{t}(\. ),U_{z}(\.)\}_{su(2)}\in \sm (\T, su(2))
\end{align}

\subsection{Local theory of the model}

The heavy lifting for this section was done in \cite{Krik2001}, but we will refer to \S 8 of \phd. 
Assuming that $U(\. ) $ is small enough, the smallness depending on the constants of the $\widetilde{RDC}$ condition,
and assuming a small conjugation $B(\. ) = \{ B_{t}(\. ),B_{z}(\.)\}_{su(2)}$, the local reduction equation is
\begin{align}
    B_{t} (x+\a) + B_{t}(x) &= -U_{t}(x) \\
    e^{-2i\pi x}B_{z} (x+\a) - \bar{B}_{z}(x) &= -U_{z}(x)
\end{align}

The difference between the equations analyzed in the reference, eqs. (8.1, 8.2) of \phd\, are the plus sign in the first one, which results
in the different arithmetic condition and in the absence of obstruction, and the presence of a complex conjugate in
the second equation, which does not change anything significant. A proof analogous to the one of theorem 8.1 of \phd\,
gives a cocycle as the one in the conclusion of proposition \ref{mainprop}, since diagonal constants can be absorbed by the
absence of obstruction in the diagonal direction.

This concludes the proof of the proposition \ref{mainprop}.

\subsection{Erratum regarding \S 8 of \phd}

It seems that the counting of codimensions in corollary 8.2 and theorem 8.3 of \phd \, did not take into account the
action of conjugations by diagonal constants on the obstruction $P$, which should lower the codimension as stated
in the theorems by $1$. For example, the obstruction of degree $1$ cocycles in $SU(2)$, whose normal form is
$(\a, E_{1}(\.))$ is of the form $z_{0}+ z_{1}e^{2i\pi \.}$ in the non-diagonal direction, and thus of real dimension
$4$ as stated in theorem 8.4, which is wrong. Conjugation by $\mathrm{diag}(e^{-2i\pi \u /2}, e^{2i\pi \u /2})$ where
$\u = \frac{\arg z_{0}}{2\pi}$ brings the obstruction to the form $|z_{0}|+ z_{1}'e^{2i\pi .}$, which is of real dimension
$3$.

\bibliography{nikosbib}
\bibliographystyle{aomalpha}

\end{document}